\documentclass{siamltex}
\usepackage{amsfonts,latexsym}
\usepackage{epsfig}
\usepackage{amsmath}
\usepackage{graphicx}
\usepackage{amsmath}
\usepackage{color}
\usepackage{ulem}

\usepackage{algorithm,algorithmic}

\title{A rational Krylov subspace method for the computation of the matrix exponential operator}

\author{
H. Barkouki \thanks{Ecole Nationale Supérieure d'Arts et Métiers, Meknès, Morocco. Email: h.barkouki@umi.ac.ma}. \and A. H. Bentbib\thanks{Facult\'e des Sciences et Techniques-Gueliz, Laboratoire de Math\'ematiques Appliqu\'ees et Informatique, Morocco. Email: a.bentbib@uca.ma} $^*$ \and K. Jbilou \thanks{Université du Littoral, C\^ote d’Opale, batiment H. Poincarré,
50 rue F. Buisson, F-62280 Calais Cedex, France. University UM6P, Benguerir Morocco.  Email: Khalide.Jbilou@univ-littoral.fr} }
\date{}

\begin{document}
\maketitle

\begin {abstract}
The computation of approximating $e^{tA}B$, where A is a large sparse matrix and B is a rectangular matrix, serves as a crucial element in numerous scientific and engineering calculations. A powerful way to consider this problem is to use Krylov subspace methods. The purpose of this work is to approximate the matrix exponential and some Cauchy-Stieltjes functions on a block vectors $B$ of $\mathbb{R}^{n \times p}$ using a rational block Lanczos algorithm. We also derive some error estimates and error bound for the convergence of the rational approximation and finally numerical results attest to the computational efficiency of the proposed method.
\end {abstract}


\section{Introduction}
The problem of approximating the matrix function 
\begin{equation}\label{equat11}
U=f(A)B,
\end{equation}
for a given matrix B of $\mathbb{R}^{n \times p}$ and $A \in \mathbb{R}^{n \times n}$  carries significant importance across numerous applications.

We are particularly interested in the exact solution of the linear ordinary differential equation systems of the form
\begin{equation}
X^{'}(t)=A X(t) \qquad \textit{and} \qquad X(0)=B,
\end{equation}
where X(t) is the matrix exponential of the form
\begin{equation}\label{equat1}
X(t)=e^{t A} B,
\end{equation}
for a fixed constant $t \geq 0$.

This approximation is a fundamental technique in many numerical methods for solving systems of ordinary differential equations or time-dependent partial differential equations. 

Over the years, various methods have been proposed to approximate the matrix exponential to a block vector B, where $1 \leq p \ll n$. In 1978, Moler and van Loan \cite{moler}  published a paper discussing nineteen questionable methods for computing the exponential of a matrix. Subsequently, Krylov subspace methods have emerged as an important development in solving the problem (\ref{equat1}) when the matrix A is sparse and of a large size.

To summarize, approximating the matrix exponential is a fundamental technique used in solving differential equations, and various methods have been proposed for approximating the matrix exponential to a block vector B. Krylov subspace methods are particularly useful in solving large and sparse matrices.

In general, there are two types of Krylov subspace methods for evaluating (\ref{equat1}) when the matrix A is large and sparse \cite{popo}. The first class of methods involves projecting the matrix A into a much smaller subspace, applying the exponential to the reduced matrix, and then projecting the approximation back to the original large space. Several methods fall under this category, including \cite{bentbib, bot, drus1, drus2, eier, moret, saad4}. The second category of methods employs a direct approximation approach, where the matrix exponential $e^A$ is replaced with an explicitly computed rational function r such that $ r(A) \approx e^A $, and then the action of the matrix exponential is evaluated. This approach has been used in several methods, including \cite{from1, from2, lopez, tref1, tref2, wu}.

 However, both of these methods require linear system solves with shifted versions of A, and in rational Krylov methods, a linear system is typically solved per iteration. Hence, a rational Krylov iteration typically incurs higher computational costs compared to a polynomial Krylov iteration, which solely requires performing a matrix-vector product with A. The efficiency of rational Krylov methods depends on the ability to solve these linear systems quickly. Rational functions may exhibit superior approximation properties to polynomials, so the overall number of iterations required by rational Krylov methods may be smaller than that required by polynomial methods, provided that the poles of the rational functions are chosen appropriately.

The objective of this paper is to address the problem (\ref{equat1}) for $p > 1$ using a new variant of rational block Lanczos algorithm \cite{bar1, bar2, nguy23}. In general, rational Krylov methods involve  finding approximation of the form of

\begin{equation}
	r_{m-1}(A)B,
\end{equation}

where $r_{m-1}$ represents a rational function with a numerator and denominator of degree m-1. The paper provides theoretical analysis by deriving error estimates and a priori error bound specifically for the rational block Lanczos method. Importantly, these results have broad applicability beyond the specific task of approximating the matrix exponential, extending to various other contexts, particularly in the approximation of Cauchy-Stieltjes functions.

 The paper is organised as follows. In section 2 we recall some basic results on Krylov polynomial approximation of the exponential. In section 3 we derive a rational block Lanczos algorithm and give some algebraic properties. Krylov rational approximation and some theoretical analysis are also derived for the matrix exponential. In section 4 we give an adaptive choice of interpolation points. Section 5 describes the application of the proposed process to the approximation of some matrix functions. The last section is devoted to some numerical experiments which attest to the computational efficiency of the proposed method.

\section{The Krylov subspace approach}
This section reviews the problem of approximating the matrix $e^A B$ using polynomial approximation, see \cite{saad5, saad4} for more details. The goal is to find an approximation of the form 
\begin{equation}
e^{A}B \approx p_{m-1}(A)B,
\end{equation}
 where $p_{m-1}$ is a polynomial of degree m-1. This approximation belongs to the Krylov subspace $K_m$, which is defined as 
 
 $$K_m :=K_m(A,B)={\rm Range} \lbrace B, AB, . . . ,A^{m-1}B \rbrace,$$

To solve this problem, the task is to find an element of $K_m$ that approximates $X = e^A B$.

 Two well-known
algorithms for finding such approximations are Arnoldi and Lanczos ones.

\subsection{Exponential approximation using the Lanczos algorithm}
 A well known algorithm for building a convenient basis of $K_m$ is  the Lanczos algorithm.
This process starts with two block vectors $V_1$ and $W_1$ and generates a bi-orthogonal basis of the subspaces $K_m(A, B)$ and $K_m(A^{T} ,C)$.
 such that 
$$K_{m}(A,B)={\rm Range} ( V_{1},V_{2},\ldots,V_{m} ).$$
and 
$$K_{m}(A^{T},C)={\rm Range} ( W_{1},W_{2},\ldots,W_{m} ).$$

The matrices $V_i$ and $W_j$ obtained through the block Lanczos algorithm exhibit biorthogonality conditions, i.e.
\begin{equation}\label{equa2.3}
\left\{ 
\begin{array}{c c c}
W_{j}^{T} V_{i}=0_{p}, \ \textit{if} \ i \neq j,\\
 W_{j}^{T} V_{i}=I_{p}, \ \textit{if} \ i = j.\\
\end{array}
\right. 
\end{equation}

Next, we give a stable version of the nonsymmetric block Lanczos algorithm that was defined in \cite{bai1}. The algorithm is summarized as follows.

\begin{algorithm}[h]
  \caption{The nonsymmetric block Lanczos algorithm ({\tt NBLA})  }
  \label{alg1}
$\bullet$ \textbf{Inputs:} $A\in \mathbb{R}^{n \times n}, B, C \in \mathbb{R}^{n \times p}$ and m an integer.\\
\begin{enumerate}
\item \textbf{Compute the QR decomposition of $C^{T} B$, i.e.,} $C^{T} B= \delta \beta $;\\
 $V_{1}=B \beta^{-1}; W_{1}=C \delta; \widetilde{V}_{2}=AV_{1}; \widetilde{W}_{2}=A^{T} W_{1};$\\
\item \textbf{For} $j=1,\ldots,m$ \\
\quad $\alpha_{j}=W_{j}^{T} \widetilde{V}_{j+1}; \widetilde{V}_{j+1}=\widetilde{V}_{j+1}-V_{j} \alpha_{j}; \widetilde{W}_{j+1}=\widetilde{W}_{j+1}-W_{j} \alpha_{j}^{T};$\\
\quad  \textbf{Compute the QR decomposition of $\widetilde{V}_{j+1}$ and $\widetilde{W}_{j+1}$, i.e.,}\\
\quad $\widetilde{V}_{j+1}=V_{j+1} \beta_{j+1}; \widetilde{W}_{j+1}=W_{j+1} \delta_{j+1}^{T};$\\
\quad  \textbf{Compute the singular value decomposition of $W_{j+1}^{T} V_{j+1}$, i.e.,}\\
\quad $W_{j+1}^{T}V_{j+1}=U_{j} \Sigma_{j} Z_{j}^{T}$;\\
\quad $\delta_{j+1}=\delta_{j+1} U_{j} \Sigma^{1/2}_{j}; \beta_{j+1}= \Sigma^{1/2}_{j} Z_{j}^{T} \beta_{j+1}$;\\
\quad $V_{j+1}=V_{j+1}Z_{j} \Sigma^{-1/2}_{j}; W_{j+1}=W_{j+1}U_{j} \Sigma^{-1/2}_{j};$\\
\quad $\widetilde{V}_{j+2}=AV_{j+1}-V_{j}\delta_{j+1}; \widetilde{W}_{j+2}=A^{T}W_{j+1}-W_{j}\beta_{j+1}^{T};$
\item \textbf{end For}.
\end{enumerate}
\end{algorithm}

Setting $\mathbb{V}_{m}=[V_{1},V_{2},\ldots,V_{m} ]$ and $\mathbb{W}_{m}=[ W_{1},W_{2},\ldots,W_{m} ]$, we have the following block Lanczos relations
$$A\mathbb{V}_{m}=\mathbb{V}_{m}A_{m}+V_{m+1} \beta_{m+1}E_{m}^{T},$$
and
$$A^{T}\mathbb{W}_{m}=\mathbb{W}_{m}A_{m}^{T}+W_{m+1} \delta_{m+1}^{T}E_{m}^{T},$$
where $E_m$ is last $mp\times p$ block of the identity matrix $I_{mp}$ and $A_m$ is the $mp \times mp$ block tridiagonal matrix defined by
\[
A_{m}=\left( 
\begin{array}{cccccc}
\alpha_{1} & \delta_{2}&  &  & \\ 
\beta_{2} & \alpha_{2} & . &  &  \\ 
 &  .& . & .&  \\ 
 &  &  .& . &  \delta_{m} \\ 
 &  &  & \beta_{m}& \alpha_{m}%
\end{array}%
\right).
\]
Because the biorthogonality condition of the matrices $\mathbb{V}_{m}$ and $\mathbb{W}_{m}$ we have $A_m=\mathbb{W}_{m}^{T}A \mathbb{V}_{m}$ and as a result $A_m$ represents the projection of the linear transformation $A$ to the subspace $K_m$. 

The matrix $X_{m}=\mathbb{V}_{m}\mathbb{W}_{m}^{T}e^{A}B$ is the projection of $e^{A}B$ on $K_m$, i.e., it is the closest approximation to $e^{A}B$ from $K_m$. 

Let $\beta \in \mathbb{R}^{p \times p}$ defining in algorithm 1, it follows immediately that
$$\mathbb{V}_{m}\mathbb{W}_{m}^{T}e^{A}B=\mathbb{V}_{m}\mathbb{W}_{m}^{T}e^{A}V_1\beta=\mathbb{V}_{m}\mathbb{W}_{m}^{T}e^{A}\mathbb{V}_{m}E_1\beta,$$
the optimal fit is therefore $X_{m}=\mathbb{V}_{m}Y_{m}$ in which $Y_{m}=\mathbb{W}_{m}^{T}e^{A}\mathbb{V}_{m}E_1\beta$. Regrettably, the computation of $Y_{m}$, which involves $e^{A}$, is not feasible in practice. Instead, an alternative approach will be adopted, which involves approximating $\mathbb{W}_{m}^{T}e^{A}\mathbb{V}_{m}$ with $e^{A_m}$, giving the approximation $Y_{m} \approx e^{A_m}E_1\beta$ and,
\begin{equation}
e^{A}B \approx \mathbb{V}_{m} e^{A_m}E_1\beta = \mathbb{V}_{m} e^{A_m} \mathbb{W}_{m}^{T}B.
\end{equation}
Now, we face the challenge of efficiently computing the block vector $e^{A_m}E_1\beta$, which is similar to the original problem we began with, but generally of significantly smaller dimensions.

Since $\mathbb{W}_{m}^{T}(tA) \mathbb{V}_{m} = tA_m$ and the Krylov subspaces associated with A and tA are identical, we can also write

\begin{equation}
e^{tA}B \approx X_m(t)=\mathbb{V}_{m} e^{tA_m}E_1\beta,
\end{equation} 
for an arbitrary scalar t.
\section{Rational approximation}
In this section, we address the rational approximation of $e^{tA}B$ using a rational block Lanczos method.
\subsection{Rational block Lanczos algorithm}
A common characteristic shared by rational Krylov methods used to compute $e^{t A} B$ is that the approximation at iteration m takes the form of $r_m (tA)B$,
 where $r_m=\dfrac{p_{m-1}}{q_{m-1}}$ is a rational function with a predetermined denominator polynomial $q_{m-1} \in \mathcal{P}_{m-1}$. In the following, we will assume that $q_{m-1}$ is factorized as

\begin{equation}
q_{m-1}(z)=\prod_{k=1}^{m-1} (1-z / \sigma_{k}),
\end{equation}

where the poles  $\sigma_1, \sigma_2, \ldots$ are numbers in the extended real plane $\bar{\mathbb{R}} := \mathbb{R} \cup \lbrace \infty \rbrace $
different
from zero.  The rational Krylov space of order m associated with (A,B) is deﬁned as 

\begin{equation}
\mathcal{K}_{m}(A,B,\Sigma_m)   =  q_{m-1} (A)^{-1} {\rm Range} \lbrace B,A B,\ldots,A^{m-1} B \rbrace.
\end{equation}

If all the poles $\sigma_1, \sigma_2, \ldots, \sigma_{m-1}$ are
 set to inﬁnity, then $q_{m-1}=1$ and the rational Krylov space  reduces to a polynomial Krylov space $K_{m}(A,B)$.
 
 This variant of rational Krylov subspace is proposed by S. Guttel \cite{guttel1}  using
 a definition based on polynomial Krylov space. This result allows to easily extend all the findings related to polynomial Krylov spaces to the rational case. In \cite{guttel1, guttel2, guttel3}, Guttel used a rational Arnoldi method for approximating $f(A)B$, specifically $e^{t A}B$. It involves obtaining an approximation of the form $r_{m} (A)B$ from a rational Krylov subspace  $\mathcal{K}_{m}(A,B,\Sigma_m)$. The result is extended her to derive a rational block Lanczos algorithm used to construct approximation of  $e^{t A}B$. This process constructs  bi-orthonormal bases of the union of block Krylov subspaces $\mathcal{K}_{m}(A,B,\Sigma_m)$ and $\mathcal{K}_{m}(A^{T},C,\Sigma_m)$. 
where $\Sigma_m=\{\sigma_{1},\ldots,\sigma_{m-1}\}$ is the set of interpolation points.

The rational block Lanczos algorithm is defined as follows
\newpage

\begin{algorithm}[h]
  \caption{The rational block Lanczos algorithm ({\tt RBLA})   }
  \label{alg2}
  \begin{enumerate}
\item \textbf{Input:} $A\in\mathbb{R}^{n \times n},\ B,\ C^{T}\in\mathbb{R}^{n \times p}$. 
\item \textbf{Output:} two biorthogonal matrices $\mathbb{V}_{m+1}$ and $\mathbb{W}_{m+1}$ of $\mathbb{R}^{n \times (m+1)p}$.\\
\item \textbf{Set} $S_{0}=B$\  {\rm and} \ $R_{0}=C$   
\item \textbf{Set} $S_{0}=V_{1}H_{1,0}$ \  {\rm and}  \  $R_{0}=W_{1}G_{1,0}$ such that $W_{1}^{T}V_{1}=I_{p}$;
\item \textbf{Initialize:} $\mathbb{V}_{1}=[V_{1}]$ \  {\rm and}  \  $\mathbb{W}_{1}=[W_{1}]$.
\item \textbf{For} $k=1,\ldots,m$
\item \ \ \ \ \ \ \textbf{if} $\lbrace\sigma_{k} = \infty\rbrace$; $S_{k}=AV_{k}$ \   {\rm and}  \ $R_{k}=A^{T}W_{k};$ \textbf{else}
\item \ \ \ \ \ \  $S_{k}=(I_n-A/ \sigma_{k})^{-1} AV_{k}$ \   {\rm and}  \ $R_{k}=(I_n-A/ \sigma_{k})^{-T} A^{T}W_{k};$  \textbf{endif}
\item \ \ \ \ \ \  $H_{k}=\mathbb{W}_{k}^{T}S_{k} $ \  {\rm and} \ $G_{k}=\mathbb{V}_{k}^{T}R_{k} $;
\item \ \ \ \ \ \ $S_{k}=S_{k}-\mathbb{V}_{k}H_{k}$ \  {\rm and} \ $R_{k}=R_{k}-\mathbb{W}_{k}G_{k};$
\item \ \ \ \ \ \ $S_{k}=V_{k+1}H_{k+1,k}$ \  {\rm and} \ $R_{k}=W_{k+1}G_{k+1,k};$ \ \ \ \  (QR factorization)
\item \ \ \ \ \ \ $W_{k+1}^{T}V_{k+1}=P_{k}D_{k}Q_{k}^{T}$; \ \  \ \ (Singular Value Decomposition)
\item \ \ \ \ \ \ $V_{k+1}=V_{k+1}Q_{k}D_k^{-1/2}$ \  {\rm and} \ $W_{k+1}=W_{k+1}P_{k}D_k^{-1/2}$;
\item \ \ \ \ \ \ $H_{k+1,k}=D_k^{1/2}Q_{k}^{T}H_{k+1,k}$ \  {\rm and} \ $G_{k+1,k}=D_k^{1/2}P_{k}^{T}G_{k+1,k}$;
\item \ \ \ \ \ \ $\mathbb{V}_{k+1}=[\mathbb{V}_{k},V_{k+1}]; \  \mathbb{W}_{k+1}=[\mathbb{W}_{k},W_{k+1}]$; 
\item  \textbf{endFor}.
\end{enumerate}
\end{algorithm}

In our setting, it is important to note that we do not have the sequence of shifts $\sigma_1, \sigma_2, \ldots, \sigma_m$ provided initially. Therefore, we need to incorporate a procedure that generates this sequence dynamically during the iterations of the process. The upcoming sections will define an adaptive procedure that ensures the generation of the sequence, while also ensuring that zero is not included as a shift. Nevertheless, this limitation does not pose a practical obstacle because we have the flexibility to select a real number $\sigma \in \mathbb{R}$ that is distinct from all the shifts. By utilizing the shifted operator $A- \sigma I$ and shifted poles $\sigma_k - \sigma $ (where the poles are non-zero), we can effectively execute the rational block Lanczos algorithm.

Let $\mathbb{H}_{m}$ and $\mathbb{G}_{m}$ be the $mp \times mp$ block upper Hessenberg matrices whose nonzeros block entries are defined by Algorithm 2, and $\widetilde{\mathbb{H}}_{m}, \widetilde{\mathbb{G}}_{m}, \widetilde{\mathbb{K}}_{m}$ and $\widetilde{\mathbb{L}}_{m}$ are the $(m+1)p \times mp$ matrices defined as
$$\widetilde{\mathbb{H}}_{m}= \left(\begin{array}{c}
\mathbb{H}_{m} \\
H_{m+1,m}E_{m}^{T} \\
\end{array}\right) \ \ , \ \ \widetilde{\mathbb{K}}_{m} \left(\begin{array}{c}
I_{mp}+\mathbb{H}_{m}\mathbb{D}_{m} \\
H_{m+1,m} \sigma_{m}^{-1} E_{m}^{T} \\
\end{array}\right),$$ 
and
$$\widetilde{\mathbb{G}}_{m}= \left(\begin{array}{c}
\mathbb{G}_{m} \\
G_{m+1,m}E_{m}^{T} \\
\end{array}\right) \ \ \textit{and} \ \ \widetilde{\mathbb{L}}_{m} \left(\begin{array}{c}
I_{mp}+\mathbb{G}_{m}\mathbb{D}_{m} \\
G_{m+1,m} \sigma_{m}^{-1} E_{m}^{T} \\
\end{array}\right),$$
where $\mathbb{D}_{m}$ is the diagonal matrix diag($\sigma_{1}^{-1}I_p,\ldots,\sigma_{m}^{-1}I_p$)  and $\lbrace \sigma_{1},\ldots,\sigma_{m} \rbrace$ are the set of interpolation points used in Algorithm 2. The following result gives some algebraic equations obtained from  the  rational block Lanczos algorithm.

 \begin{theorem}\label{theo}
Suppose we have matrices $\mathbb{V}_{m+1}$ and $\mathbb{W}_{m+1}$, which are generated using the rational block Lanczos algorithm (Algorithm 2), specifically for the extra interpolation points at $\sigma_{m}=\infty$. We assume that matrix A is nonsingular and that all the interpolation points $\sigma_{i}, \, i=1,\ldots,m-1$, are finite real numbers that are not equal to zero. Then, it can be shown that there exist block upper Hessenberg matrices $\widetilde{\mathbb{H}}_{m}, \widetilde{\mathbb{G}}_{m}, \widetilde{\mathbb{K}}_{m}$  and $\widetilde{\mathbb{L}}_{m}$, each with dimensions of $(m+1)p \times mp$. These matrices satisfy the following relationships for the left and right Krylov subspaces: 	
 	
 \begin{eqnarray}
A\mathbb{V}_{m} \mathbb{K}_{m} & = & \mathbb{V}_{m+1} \widetilde{\mathbb{H}}_{m} \label{equa3.5} \\
A^{T}\mathbb{W}_{m}{}\mathbb{L}_{m} & = & \mathbb{W}_{m+1} \widetilde{\mathbb{L}}_{m}, \label{equa3.7}  \\
A_{m} & = & \mathbb{W}_{m}^{T} A \mathbb{V}_{m} = \mathbb{H}_{m} \mathbb{K}_{m}^{-1} , \label{equa3.88} 
\end{eqnarray}
where $\mathbb{H}_{m}$, $\mathbb{G}_{m}$, $\mathbb{K}_{m}$ and $\mathbb{L}_{m}$ are the $mp \times mp$ block upper Hessenberg matrices obtained by deleting the last row block vectors of  $\widetilde{\mathbb{H}}_{m}, \widetilde{\mathbb{G}}_{m}, \widetilde{\mathbb{K}}_{m}$  and $\widetilde{\mathbb{L}}_{m}$, respectively.
\end{theorem}
\begin{proof}
In the ﬁrst iteration k = 1 we use the initial block vectors B and C to construct $V_1$ and $W_1$, which are the biorthogonal first block vectors of $\mathbb{V}_{m+1}$ and $\mathbb{W}_{m+1}$.

In the following iterations, we have 
 $$V_{k+1}H_{k+1,k}=(I_n-A/ \sigma_{k})^{-1} AV_{k}-\mathbb{V}_{k}H_{k}, \ k=1,\ldots,m$$
which can be written as
\begin{equation}\label{equa3.11}
[\mathbb{V}_{k} \ V_{k+1}]\left [ \begin{array}{c}
H_{k} \\
H_{k+1,k} \\
\end{array}\right]=(I_n-A/ \sigma_{k})^{-1} AV_{k}. \ \ k=1,\ldots,m
\end{equation}
Multiplying (\ref{equa3.11}) on the left by $(I_n-A/ \sigma_{k})$ and replacing $V_{k}$ by $\mathbb{V}_{k}E_{k}$ gives 
$$(I_n-A/ \sigma_{k}) \mathbb{V}_{k+1}\left[ \begin{array}{c}
H_{k} \\
H_{k+1,k} \\
\end{array}\right] =A \mathbb{V}_{k}E_{k},$$
where $E_{k}$ is an $kp \times p$ tall thin matrix with an identity matrix of dimension p at the $k^{th}$ block and zero elsewhere. Re-arranging the expression of the last equation as
\begin{equation}\label{equa3.12}
A \mathbb{V}_{k+1} \left (\left[ \begin{array}{c}
E_{k} \\
0 \\
\end{array}\right]+\sigma_{k}^{-1} \left[ \begin{array}{c}
H_{k} \\
H_{k+1,k} \\
\end{array}\right] \right)\
 = \mathbb{V}_{k+1} 
 \left [ \begin{array}{c}
H_{k} \\
H_{k+1,k} \\
\end{array}\right]   , k=1,\ldots,m.
\end{equation}

which gives, after m iterations, the following relation

\begin{equation}
A \mathbb{V}_{m}(I_{mp}+\mathbb{H}_{m}\mathbb{D}_{m})+AV_{m+1}H_{m+1,m} \sigma_{m}^{-1}E_{m}^{T} = \mathbb{V}_{m} \mathbb{H}_{m}+H_{m+1,m} V_{m+1}E_{m}^{T}, 
\end{equation}
finally we have
\begin{equation}
A\mathbb{V}_{m+1} \widetilde{\mathbb{K}}_{m}  =  \mathbb{V}_{m+1} \widetilde{\mathbb{H}}_{m}.
\end{equation}

If one chose $\sigma_{m}$ such that $\sigma_{m}=\infty$, then the last row block vectors of $\widetilde{\mathbb{K}}_{m}$ is zero and consequently the relation (\ref{equa3.5}) is satisfied. In a similar way, the relation (\ref{equa3.7}) can be shown.

The equation (\ref{equa3.88}) can be obtained using the biorthogonality  condition and the relation (\ref{equa3.5}).
\end{proof}

\subsection{Rational approximation}
In this section, we consider rational block Lanczos algorithm described in the previous section for approximating the matrix exponential operation using rational function of the form
\begin{equation}
e^{tA}B \approx r_{m-1}(tA)B.
\end{equation}
Let $\mathbb{V}_m$ and $\mathbb{W}_m$ be the bi-orthogonal basis computed by the rational block Lanczos algorithm (Algorithm 2). The rational Lanczos approximation for $e^{tA}B$ is defined as
\begin{equation*}
X_m(t)=\mathbb{V}_m e^{t A_m} \mathbb{W}_{m}^{T}B, \  \   \textit{where} \  \ A_m= \mathbb{W}_{m}^{T} A  \mathbb{V}_{m}.
\end{equation*} 

Using the fact that $B=V_1 H_{1,0} $ and the bi-orthogonality condition gives
\begin{equation} \label{rationapp}
X_m(t)=\mathbb{V}_m e^{t A_m} E_{1} H_{1,0},
\end{equation} 

where $E_1$ corresponds to the first p columns of the identity matrix $I_{mp}$.

 Since the matrix exponential $X(t)=e^{tA}B $ is the exact solution of the linear ordinary differential equation systems of the form
\begin{equation}
X^{'}(t)=A X(t) \qquad \textit{and} \qquad X(0)=B,
\end{equation}
we define the residual error for the approximation $X_m(t)$ as

\begin{equation}\label{res}
R_m(t)= AX_m(t) - X_{m}^{'}(t),
\end{equation}

and the exact error as
\begin{equation}\label{err}
E_m(t)= X(t) - X_{m}(t).
\end{equation}

The following simple theorem provides an explicit expression for the residual $R_m(t)$ that avoids computing matrix products with the large matrix A. 

\begin{theorem} \label{theores}
 Let $\mathbb{V}_{m}$ and $\mathbb{W}_{m}$  be the matrices generated by Algorithm 2  and $X_m(t) = \mathbb{V}_{m} e^{t A_m} E_{1} H_{1,0}$ be the approximation to $X(t) = e^{tA} B$. Then for any $t \geq 0$ the residual expression can be expressed as :
\begin{equation} \label{theo1}
 R_m(t)=V_{m+1} H_{m+1,m}E_{m}^{T}\mathbb{K}_{m}^{-1} e^{t A_m} E_{1} H_{1,0}
\end{equation}   
\end{theorem}
\begin{proof}
 It follows from the expression of the approximation $X_m(t)$ that 
 $$X_m^{'}(t)=\mathbb{V}_{m}A_m e^{t A_m} E_{1} H_{1,0} $$
Using relations (\ref{equa3.5}) and (\ref{equa3.88}), we get
\begin{eqnarray*}
AX_m(t) & = & A \mathbb{V}_{m} e^{t A_m} E_{1} H_{1,0}, \nonumber \\
& = & \mathbb{V}_{m} A_m e^{t A_m} E_{1} H_{1,0} + V_{m+1} H_{m+1,m}E_{m}^{T}\mathbb{K}_{m}^{-1} e^{t A_m} E_{1} H_{1,0} ,  \nonumber \\
\end{eqnarray*}

which gives the following result :

\begin{equation}\label{relatres}
R_m(t)=AX_m(t)-X_{m}^{'}(t)=V_{m+1} H_{m+1,m}E_{m}^{T}\mathbb{K}_{m}^{-1} e^{t A_m} E_{1} H_{1,0}.
\end{equation}

\end{proof}

The above result shows that the exponential residual  $R_m(t)$ can be computed at no additional cost. Then we can use $R_m(t)$ instead of the unknown exact error $E_m(t)$ in the stopping criterion to decide when the Krylov approximation (\ref{rationapp}) is to be considered suffciently accurate.

In the following theorem we establish an error bounds for the rational block Lanczos approximation of $e^{tA}B$.

We first introduce the logarithmic 2-norm of a matrix A  defined by
$$\mu _2(A)=\lambda_{max} \left( \dfrac{A+A^{T}}{2} \right).$$

The logarithmic norm is used to bound the norm of the matrix exponential such that

$$\Vert  e^{tA} \Vert  \leq e^{t \mu _2(A)} \ \  \  t \geq 0,$$

see \cite{ Dat} for more details.

\begin{theorem} \label{theo_err_bound}
  Under the above assumptions, the error in the rational block Lanczos approximation  of $e^{tA}B$ is bounded in the following way :
\begin{equation} \label{bound}
  \Vert X(t) - X_m(t) \Vert \leq  \max_{s \in [0, t]} \Vert  R_{m} (s)  \Vert  \dfrac{e^{t \mu _2(A)} -1}{\mu _2(A)} .
\end{equation}
\end{theorem}

\begin{proof}
According to relation (\ref{relatres}) we have,

$$X_m^{'}(t)=AX_m(t)-R_m(t),$$
then

\begin{eqnarray*}
E_m^{'}(t) & = & X'(t) - X_m^{'}(t), \nonumber \\
& = &  A e^{t A} B - AX_m(t)+ R_m(t),  \nonumber \\
& = &  A E_m(t) + R_m(t)  \nonumber \\
\end{eqnarray*}
Consider the initial condition 
$$E_m(0)=X(0) - X_m(0)=B - \mathbb{V}_m \mathbb{W}_{m}^{T} B=0,$$
and solve the initial value problem for $E_m(t)$, we obtain

$$E_m(t)=\int _{0}^{t} e^{(t-s)A}  R_m(s) ds,$$
since $t-s > 0$, we have
$$\Vert  e^{(t-s)A} \Vert  \leq e^{(t-s) \mu _2(A)}  $$
which gives

\begin{eqnarray*}
\Vert E_m(t) \Vert &  \leq   &   \left\|    \int _{0} ^{t}  e^{(t-s)A}  R_m(s) ds \right\|    , \nonumber \\
& \leq & \int _{0} ^{t}  \left\| e^{(t-s)A} \right\|    \left\| R_m(s) \right\| ds ,  \nonumber \\
& \leq & \max _{s \in [0, t]} \left\|   R_m(s) \right\|   \int _{0} ^{t}  e^{(t-s) \mu_2 (A)} ds    \nonumber \\
& \leq &   \max _{s \in [0, t]}   \left\| R_m(s) \right\|   \dfrac{e^{t \mu _2  (A)} -1}{\mu _2 (A)}    \nonumber \\
\end{eqnarray*}

 \end{proof}

{\bf Numerical example.} In this experiment, we present some  numerical examples to demonstrate the error
bound obained in thorem \ref{theo_err_bound}. We will construct two testing matrices and plot the approximation error $\Vert X(t) - X_m(t) \Vert $ against our error bound (\ref{bound}) as a function of the number of iterations.

For the first test, we consider the $1600 \times 1600$ diagonal matrix $A = A_1$ of the logarithm
of equispaced values between 0.2 and 0.99. This can be defined as {\tt  $ A_1$
= diag(log(linspace(0.2,0.99,100)))} in Matlab notation. The spectrum of A is contained in the interval
[-1.61, -0.0101].

For the second one, $A = A_2 \in \mathbb{R}^{n \times n}$ is the block diagonal matrix  with $2 \times 2$ blocks of the form  

$$\left [ \begin{array}{ c  c}
a_i    &  c \\
c   &   a_i \\
\end{array}\right],$$

where $c=1/2, a_i=(2i - 1)/(n+1)$ for $i=1,\ldots, n/2$ and n=1600.

The obtained plots for this experiment are given in Figures \ref{fig0.1} and \ref{fig0.2}.

\begin{figure}[!h]
 \begin{center}
\includegraphics[height=3in ,width=2.75in]{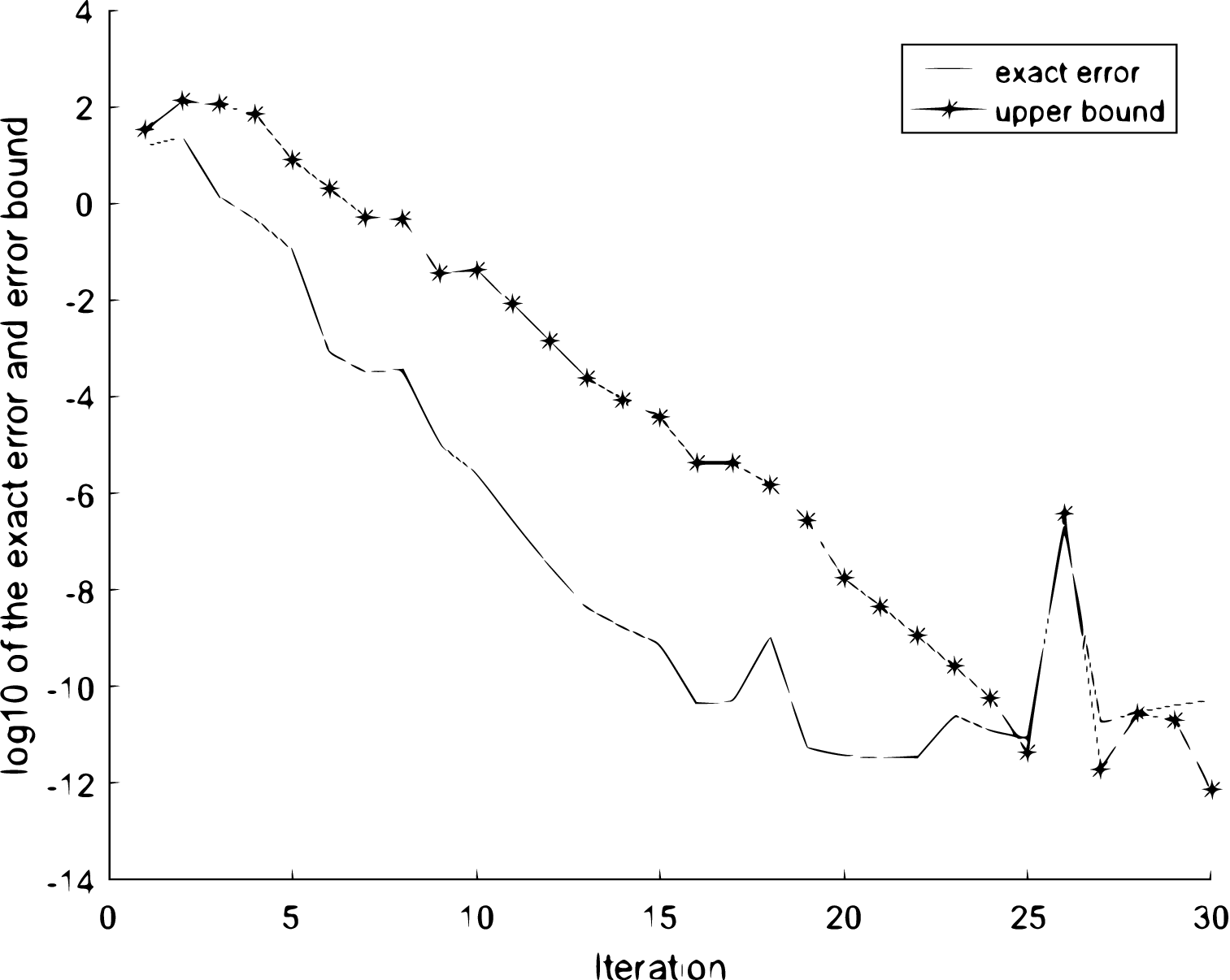}
\includegraphics[height=3in ,width=2.75in]{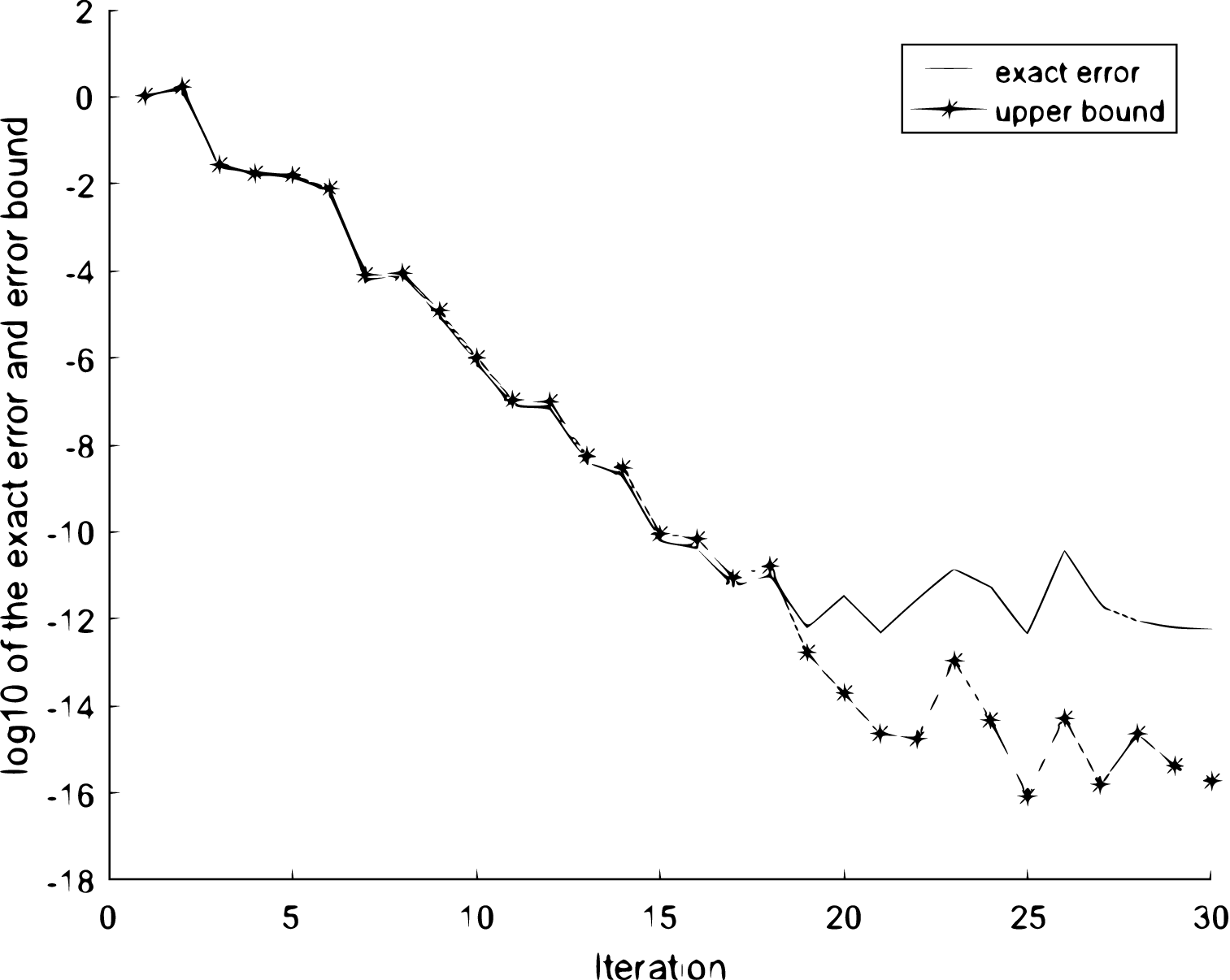}
 \end{center}
\caption{ The exact error (blue line) and the error bound (red line) as a function of iteration number m  using the matrix {\tt A = $A_1$} for $t=1$ (left plot) and $t=10^{-2}$ (right plot).}\label{fig0.1}
\end{figure}

\begin{figure}[!h]
 \begin{center}
\includegraphics[height=3in ,width=2.75in]{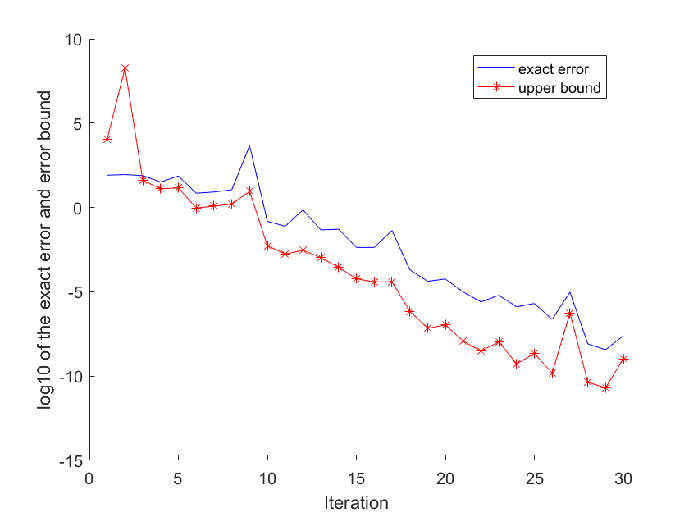}
\includegraphics[height=3in ,width=2.75in]{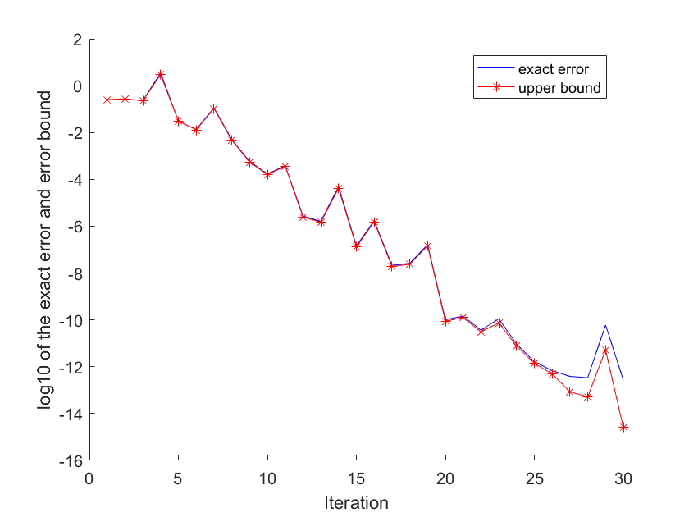}
 \end{center}
\caption{ The exact error (blue line) and the error bound (red line) as a function of iteration number m  using  the matrix {\tt A= $A_2$} for $t=1$ (left plot) and $t=10^{-2}$ (right plot).}\label{fig0.2}
\end{figure}

\section{Adaptive pole selection}
 Rational Krylov methods are generally known for providing more accurate approximations compared to polynomial Krylov methods. However, the selection of appropriate shifts $\sigma_j$ is not a straightforward automated process. It requires careful consideration to ensure a more precise rational Krylov subspace.

In prior research (references \cite{guger3, guger2}), a specific algorithm called the Iterative Rational Krylov Algorithm (IRKA) was proposed within the context of ${\cal H}_{2}$-optimal model-order reduction. The IRKA algorithm offers a method for selecting interpolation points $\sigma_{i}$, $i=1,\ldots,m$. The IRKA begins with an initial set of interpolation points and utilizes them to derive a reduced-order system. Subsequently, a new set of interpolation points is determined as the Ritz values $-\lambda_{i}(A_m), i=1,\ldots,m$, where  $\lambda_{i}(A_m)$ are the eigenvalues of $A_m$. The process iterates until the Ritz values obtained from consecutive reduced-order models reach a point of stagnation, indicating convergence.


Several techniques, as described in references \cite{boden4, grim15,lee22, soppa25}, have been proposed to address the challenge of selecting appropriate interpolation points. These methods focus on constructing the next interpolation point at each step and operate under the principle that the shifts should be chosen in a way that minimizes the norm of specific error approximations at every iteration. The goal of these techniques is to identify interpolation points that contribute to reducing the errors in the approximation process. By minimizing the norm of the approximation errors, the methods aim to enhance the accuracy and quality of the computed results during each iteration.

Druskin and Simoncini \cite{drus7} have introduced an alternative adaptive strategy for selecting the poles  in the rational Arnoldi approximation of the resolvent function. This approach capitalizes on the interpolation properties of the rational Arnoldi method, which also hold true for the rational Lanczos method. 

This approach is based on a representation of the rational residual as

\begin{equation}\label{residu}
\dfrac{r_m(A)B}{r_m(s)}, \qquad r_m(z)=\prod _{j=1}^{m} \dfrac{z- \lambda _j}{z-\sigma _j},
\end{equation}
where $\lambda _j$ are the rational Ritz values $\Lambda (A_m)$.

At iteration m of this method the poles $\sigma_1,\ldots,\sigma_m$ have already been chosen. The next interpolation point, denoted as $\sigma_{m+1}$, is determined as

\begin{equation}
	\sigma_{m+1}=arg \left( \max_{s \in S} \dfrac{1}{\vert r_{m}(s) \vert} \right),
\end{equation}

 where m is the current iteration index. The specific choice of the set S, which influences the selection of $\sigma_{m+1}$, will be discussed later. The algorithm for constructing the new shift $\sigma_{m+1}$ is provided as follows.

\begin{algorithm}[h]
  \caption{The procedure for selecting the shifts}
  \label{selectpoints}
  \begin{itemize} 
\item \textbf{Input :} $ \lbrace \lambda_{j} \rbrace  _{j=1}^{m}, \lbrace \sigma_{j} \rbrace  _{j=1}^{m}$ and the set $\lbrace \eta_{1},\ldots,\eta _{l} \rbrace$;\\
\begin{enumerate}
\item For \ $ j=1,\ldots,l-1$
\item \ \ \ $\mu_{j}=\displaystyle arg \max _{\mu \in [\eta_{j},\eta_{j+1}]} \dfrac{1}{\vert r_{m}(\mu) \vert }, \ r_{m}(z)= \displaystyle  \prod_{i=1}^{m} \dfrac{z-\lambda_{i}}{z-\sigma_{i}}$;
\item end
\item $\sigma_{m+1}=\displaystyle arg \max_{j=1,\ldots,l-1} \dfrac{1}{\vert r_{m}(\mu_{j})\vert } $.
\end{enumerate}
\end{itemize}
\end{algorithm}


\medskip
The set $ \lbrace \lambda_{i} \rbrace  _{i=1}^{m}$ contains the eigenvalues of the matrix $A_m$ at the iteration $m$ and $\lbrace \sigma_{i} \rbrace  _{i=1}^{m}$ represents the previously selected interpolation points.

The set $\lbrace \eta_{1},\ldots,\eta _{l} \rbrace$ contains two given values $\sigma _{0} ^{(1)}, \sigma _{0} ^{(2)} \in \mathbb{R}$ and the previously chosen interpolation points.

By integrating the adaptive approach for selecting interpolation points with the rational block Lanczos algorithm (Algorithm \ref{alg2}), we obtain the Adaptive Rational Block Lanczos (ARBL) algorithm. This process is specifically designed for computing approximations of the matrix exponential. The ARBL algorithm combines the benefits of the rational block Lanczos method with the adaptiveness of the interpolation point selection. By incorporating this adaptive strategy, the ARBL algorithm aims to improve the accuracy and efficiency of the approximation process for the matrix exponential.
\section{Application to the approximation of Cauchy-Stieltjes functions} In science and engineering, one common challenge is computing the product of a matrix-block and a vector, denoted as f(A)B, where  $A \in \mathbb{R}^{n \times n}, B \in \mathbb{R}^{n \times p}$ and f is a function that satisfies certain conditions, such as being analytic in a neighborhood of the eigenvalues of A, denoted as $\Lambda(A)$, which are the set of values that A can multiply a vector by without changing its direction. The function f(A) is called a matrix function, and it is necessary to define it before computing the product. For more information on matrix functions and their definitions, you can refer to the book by Higham \cite{higham}.

This section aims to approximate the action of f(A) on B for certain types of Cauchy-Stieltjes functions using  rational block Lanczos algorithm (Algorithm \ref{alg2}). Cauchy-Stieltjes functions can be expressed in the form 

\begin{equation}
f(x)=\int _{0}^{\infty} \dfrac{\mu (z) }{x+z} d(z),
\end{equation}

where $\mu (z) d(z)$ is a non-negative measure on $[0,\infty]$. Some important examples of such functions are

\begin{equation*}
f_1(x)=x^{-\alpha}=\dfrac{sin(\alpha \pi)}{\pi} \int _{0}^{\infty} \dfrac{z^{-\alpha}}{x+z}  dz \qquad  0 < \alpha  <1,
\end{equation*}

\begin{equation*}
f_2(x)=\dfrac{log(1+x)}{ x}=\int _{1}^{\infty} \dfrac{z^{-1}}{x+z}  dz,
\end{equation*}
and
\begin{equation*}
f_3(x)=\dfrac{e^{-t \sqrt{x}}-1}{ x}=\int _{0}^{\infty} \dfrac{1 }{x+z} \dfrac{sin(t \sqrt{z})}{- \pi z} dz,
\end{equation*}

see \cite{drus2, drus4, higham} for aplications of these functions. Functions of this type also arise in the context of computation of
Neumann-to-Dirichlet and Dirichlet-to-Neumann maps \cite{drus5, arioli}, the solution of systems of stochastic differential equations \cite{allen}, and in quantum chromodynamics \cite{van}.

The rational Lanczos approximation of the matrix function f(A)B is defined as
\begin{equation}
f_m= \mathbb{V}_m f(A_m)\mathbb{W}_m ^{T}B=\mathbb{V}_m f(A_m) E_1 H_{1,0}, 
\end{equation}
where $B=V_1 H_{1,0}, A_m=\mathbb{W}_{m}^{T} A \mathbb{V}_m$ and  $\mathbb{V}_{m}$, $\mathbb{W}_{m}$ are the matrices generated  by Algorithm \ref{alg2}.
\section{numerical results}
In this section, we illustrate the performance of the adaptive  rational block Lanczos method (ARBL) when approximating $e^{t A} B$ for multiple
values of $t>0$ using the {\tt(ARBL)} algorithm and for approximating $f(A)B$ for some
Cauchy-Stieltjes functions. All experiments were conducted using MATLAB R2016a on a computer equipped with an Intel Core i7 at 1.8GHz and 16GB of RAM.

The system of equations with the matrix $(I_n-A/ \sigma_k)$ in Algorithm \ref{alg2}, line 8 is
solved by means LU or Cholesky factorization of $(I_n-A/ \sigma_k)$ by using the backslash $ \setminus $ operator of Matlab. 

 The purpose of the first and second experiments is to illustrate with the help of simple examples how the exact error $\Vert X(t) - X_m(t) \Vert_{\infty}$ behave in practice. The third one give the numerical resulats for the residual error derived in theorem \ref{theores}. For the last experiment, we give the numerical results for approximating $f(A)B$ for some Cauchy-Stieltjes functions.

\subsection{Examples for approximations of $e^{t A}B$}

\textbf{Example 1.} For this experiment, we plotted the exact error $\Vert X(t) - X_m(t)  \Vert_{\infty}$ as a function of time t, where $t \in [10^{-1}, 1]$ and for a fixed number of iteration m. 

For the first example, the corresponding matrix {\tt A=fdm} is obtained from the centered finite difference discretization of the operator
$$L_{A}(u)  =  \Delta u-f(x,y) \frac{\partial u}{\partial x}-g(x,y) \frac{\partial u}{\partial y}-h(x,y)u,$$
on the unit square $[0, 1] \times [0, 1]$ with homogeneous Dirichlet boundary conditions with
\begin{equation*}
\left\{ 
\begin{array}{c c c}
f(x,y) & = & e^{xy},\\ 
g(x,y) & = & sin(xy),\\
h(x,y) & = & y^2-x^2.\\
\end{array}
\right. 
\end{equation*}
The number of inner grid points in each direction was $n_0 = 40$ and the dimension of A is $n = n_{0}^{2}=1600$. 

For the second example, the matrix {\tt A = poisson} is obtained from the centered finite difference discretization of the two-dimensional Poisson operator  
$$u_{xx} + u_{yy} = f(x,y)$$
with Dirichlet boundary conditions  on  $[0, 1] \times [0, 1]$ with $f(x,y)= 1.25 \times exp(x+\dfrac{y}{2})$.  The dimension of the matrix A is $n=n_{0}^{2}=1600$.


For both examples B and C are chosen to be a random matrix with entries uniformly distributed in $[0, 1]$ and $p=3$.

The obtained plots for this experiment are given in Figures \ref{fig1.1} and \ref{fig1.2}.

\begin{figure}[!h]
 \begin{center}
\includegraphics[height=3in ,width=2.75in]{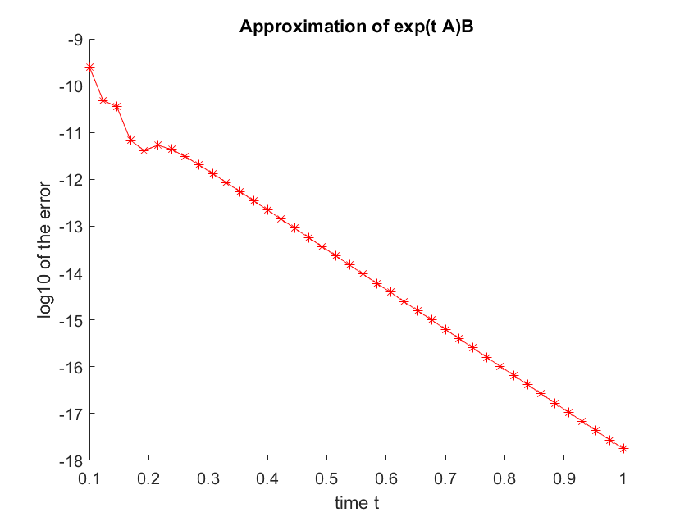}
\includegraphics[height=3in ,width=2.75in]{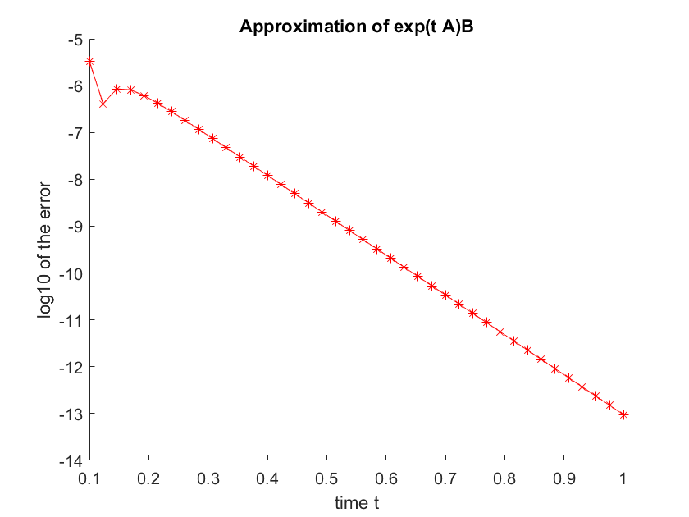}
 \end{center}
\caption{ The exact error as a function of time  using the matrix {\tt A = fdm} for $m=15$ (left plot) and $m=10$ (right plot).}\label{fig1.1}
\end{figure}

\begin{figure}[!h]
 \begin{center}
\includegraphics[height=3in ,width=2.75in]{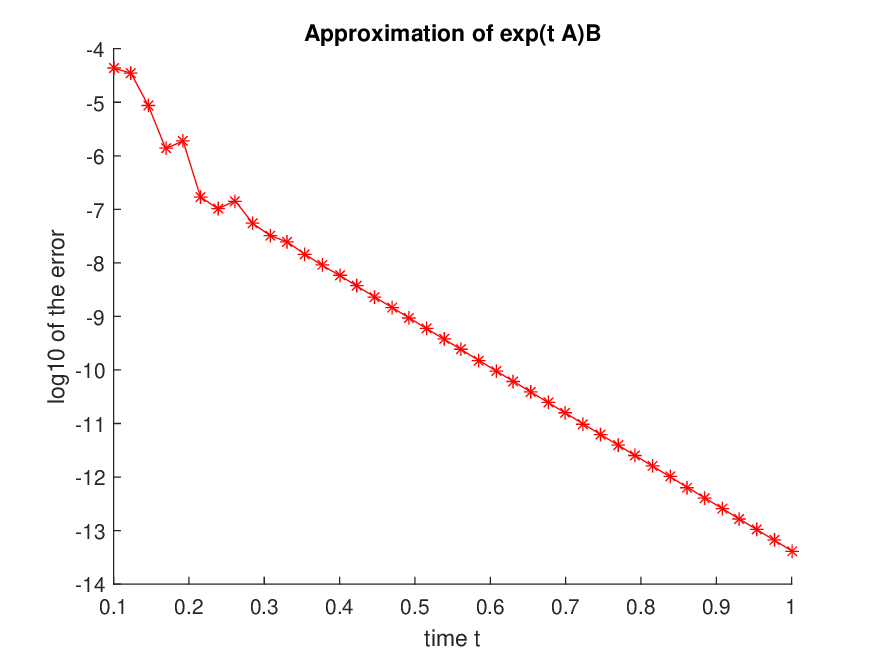}
\includegraphics[height=3in ,width=2.75in]{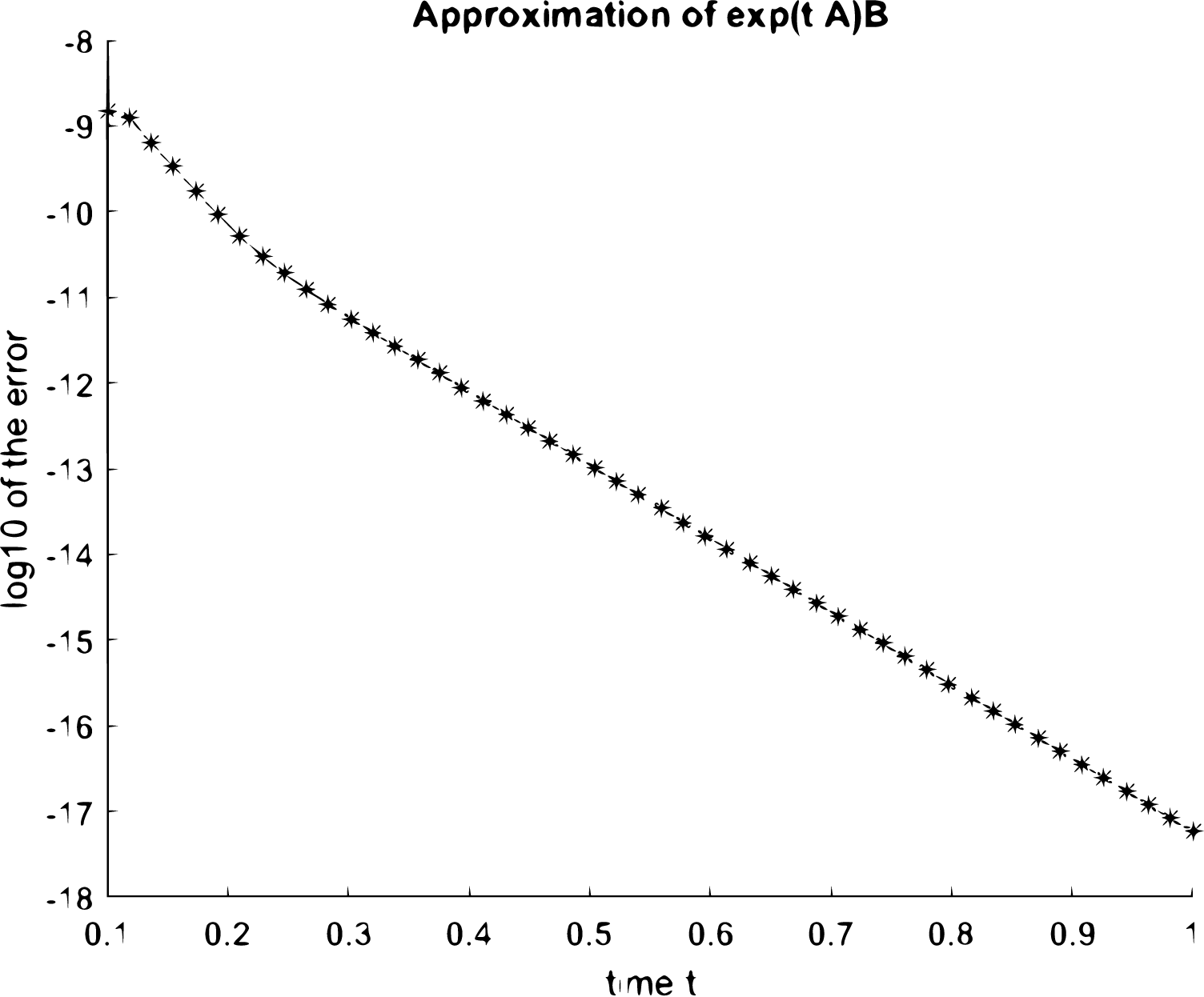}
 \end{center}
\caption{ The exact error as a function of time using the matrix {\tt A=poisson} for $m=10$ (left plot) and $m=15$ (right plot).}\label{fig1.2}
\end{figure}

\newpage

\textbf{Example 2.}  For the second experiment, we give the $H_{\infty}$ exact error for different values of iterations $m \in \lbrace 10, 15, 20, 25, 30, 40  \rbrace $ and a fixed scalar t.

For the first example, we consider the matrix {\tt A = poisson} using in the fist experiment with n=6400 ($n_0=80$). For the second example, we use the matrix {\tt A = add32} from Sparse Matrix collection.



 For both examples B and C are chosen to be a random matrix with entries uniformly distributed in $[0, 1]$. The number of block vectors of matrices B and C are  p=3 for the first example and p=4 for the second one.

The obtained results for this experiment are given in table \ref{tab1}.

\begin{table}[!h]
\begin{center}
\caption{$H_{\infty}$ exact error of $e^{tA}B$ for multiple values of dimension m }\label{tab1}
\begin{tabular}{l|c|c}
\hline
\quad \quad \quad \quad Test problem & \;\;  dim.(m) &  \;\;  $ \Vert X-X_m \Vert _{\infty} $ \\  
\hline
A={\tt poisson} &10 & $5.38 \times 10^{-15}$   \\ 
n=6400, p=3 &  20& $3.74 \times  10^{-19}$ \\

t=1&30 & $1.10 \times 10^{-19}$ \\
& 40& $3.37 \times 10^{-19}$ \\
\hline
A={\tt poisson}  &10 & $2.40 \times 10^{-20}$  \\ 
 n=6400, p=3 &  20& $1.87 \times  10^{-23}$ \\

t=2&30 & $1.03 \times 10^{-27}$\\
& 40& $4.54 \times 10^{-28}$ \\
\hline
A={\tt add32} &10 & $3.46 \times 10^{-15}$  \\ 
n=4960 , p=4 &  20& $3.14 \times  10^{-13}$ \\

$t=10^{-2}$&30 & $4.25 \times 10^{-15}$\\
& 40& $5.36 \times 10^{-14}$ \\
\hline
A={\tt add32} &10 & $3.58 \times 10^{-15}$   \\ 
n=4960 , p=4 &  20& $ 2.14 \times  10^{-14}$  \\

$t=10^{-1}$&30 & $ 1.27 \times 10^{-14}$ \\
& 40& $7.17 \times 10^{-14}$ \\
\hline

\end{tabular}
\end{center}
\end{table}

\textbf{Example 3.} 
In this experiment, we give the $H_{\infty}$ norm of the residual error given in theorem \ref{theores} for different values of time t and a fixed iteration m.

we use the  matrices: {\tt Rail3113}, {\tt Rail5177} and {\tt Rail20209}  from Benchmark collections \cite{mehrman}. We use also the matrix {\tt add32} from Sparse Matrix collection.

B is chosen to be a random matrix with entries uniformly distributed in $[0, 1]$ for  {\tt add32} matrix. For the others examples, B is constructed by {\tt Rail3113}, {\tt Rail5177} and {\tt Rail20209} problems.

For all examples, the matrix C is chosen  such that $C=B^T$.

In table \ref{tab2}, we report the space of dimension
(dim), the $H_{\infty}$ residual norm $\Vert R_{m} \Vert _{\infty}$ and the required CPU time for six values of time step t  $\in \lbrace 10^{-2}, 5.10^{-2}, 10^{-1}, 5.10^{-1}, 1, 10 \rbrace$.

\begin{table}[!h]
\begin{center}
\caption{$H_{\infty}$ residual error of $e^{tA}B$ for multiple values of time step t  }\label{tab2}
\begin{tabular}{l|c|c|c|c}
\hline
\quad \quad \quad \quad Test problem & \;\;  dim.(m) & \;\;   t  & \;\;  $ \Vert R_{m} \Vert _{\infty} $ & \;\; time (s) \\  
\hline
 & & $10^{-2}$ & $2.75 \times 10^{-9}$ &  \\ 
     &   &  $5. 10^{-2}$ & $1.79 \times 10^{-10}$ & \\
{\tt add32}&  10& $10^{-1}$& $4.69\times  10^{-8}$& 100 \\
n=4960, p=4 & & $5.10^{-1}$& $2.24 \times  10^{-8}$ & \\
& & 1 & $5.11 \times 10^{-8}$& \\
& & 10& $5.89 \times 10^{-7}$& \\
\hline
 & & $10^{-2}$ & $1.84 \times 10^{-10}$ &  \\ 
     &   &  $5. 10^{-2}$ & $1.35 \times 10^{-9}$ & \\
{\tt Rail3113}&  10& $10^{-1}$& $2.11 \times  10^{-9}$& 5.14 \\
n=3113, p=6& & $5.10^{-1}$& $2.45\times  10^{-9}$ & \\
& & 1 & $1.48 \times 10^{-10}$& \\
& & 10& $1.30 \times 10^{-9}$& \\
\hline
& & $10^{-2}$ & $2.40 \times 10^{-25}$ &  \\ 
     &   &  $5. 10^{-2}$ & $8.38 \times 10^{-27}$ & \\
{\tt Rail5177}&  20& $10^{-1}$& $5.40 \times  10^{-22}$& 33 \\
n=5177, p=6 & & $5.10^{-1}$& $1.03 \times  10^{-25}$ & \\
& & 1 & $1.44 \times 10^{-19}$& \\
& & 10& $3.49 \times 10^{-25}$& \\
\hline
& & $10^{-2}$ & $1.59 \times 10^{-32}$ &  \\ 
     &   &  $5. 10^{-2}$ & $7.43 \times 10^{-32}$ & \\
{\tt Rail20209}&  25& $10^{-1}$& $3.42 \times  10^{-33}$& 600 \\
n=20209, p=6 & & $5.10^{-1}$& $2.16 \times  10^{-26}$ & \\
& & 1 & $2.57 \times 10^{-30}$& \\
& & 10& $8.91 \times 10^{-27}$& \\
\hline
\end{tabular}
\end{center}
\end{table}

\subsection{Examples for approximations of f(A)V, for Cauchy-Stieltjes functions}
In the following examples,  we show the performance  of the  rational block Lanczos algorithm (Algorithm \ref{alg2}) when approximating $f(A)B$. We used three Cauchy-Stieltjes functions:

$$ f_1(x)=x^{-1/2}, \   \    \  f_2(x)=\dfrac{log(1+x)}{x}, \     \    \  f_3(x)=exp(x).$$

In all examples, the initial block vectors B and C are generated randomly with uniformly distributed entries in the interval [0; 1] and the block size p is 5. 

\textbf{Example 4.} In this experiment, we consider $f_1(x)=x^{-1/2}$, this function is used for the Newmann-to-Dirichlet map \cite{drus5} of the differential problem
$$\dfrac{d^2 U(t)}{dt} - AU(t) =0, \   \ t>0; \       \    U'(0)=-B \      \textit{and}     \      U ( \infty)=0,$$

where $U(0)=A^{-1/2}B$.

The matrix A was obtained from the three point central finite difference discretization (CFDD) of the following partial differential equation on the unit square $[0; 1] \times [0; 1]$ with Dirichlet homogeneous boundary conditions

\begin{equation*}
\left\{ 
\begin{array}{c c c}
\dfrac{\partial U}{\partial t}- {\cal L}_i(U) & = & 0 \ \  \textit{on} \ \  \ (0,1)^2 \times (0,1), \\ 
U(x,y,t) & = & \qquad 0 \ \ on \ \ \  \partial (0,1)^2 \ \  \forall \ t\in [0,1], \\
U(x,y,0) & = & B \ \  \forall \ x, y \in [0,1]^2. \\
\end{array}
\right. 
\end{equation*}
The number of discretization in each direction is $n_0$ and the dimension of the matrix A is $n = n_{0}^{2}$. We chose two different matrices $A_1$ and $A_2$ , which are coming from CFDD of
the following operators

\begin{eqnarray*}
{\cal L}_1(u) & = & -u_{xx}-u_{yy}, \\ 
{\cal L}_2(u) & = & -100 u_{xx} - u_{yy}+10x u_{x}.\\
\end{eqnarray*}

In table \ref{tab3}, we display the exact  error $\Vert f(A)B-f_m \Vert _{\infty}$ and the required CPU time  for the space of dimension $m \in \lbrace 20, 30, 40  \rbrace $. 
\begin{table}[!h]
\begin{center}
\caption{$H_{\infty}$ exact error of $A^{-1/2}B$ for the operators ${\cal L}_1(u), {\cal L}_2(u)$ and for multiple values of n.}\label{tab3}
\begin{tabular}{l|c|c|c|c}
\hline
\quad \quad \quad \quad Test problem & \;\; n & \;\;  dim.(m)   & \;\;  $ \Vert f(A)B-f_m \Vert _{\infty} $ & \;\; time (s) \\  
\hline
${\cal L}_1(u)$  & 3600 & 20 & $4.32\times 10^{-9}$ & 3.41\\ 
 & & 30  & $8.07 \times 10^{-10}$ & 4.54 \\
  &  & 40&  $6.25 \times  10^{-12}$& 6.96 \\
& 6400 & 20 & $5.35 \times  10^{-11}$ & 11.10 \\
& & 30 & $3.31 \times 10^{-11}$& 14.49\\
& & 40 & $2.71 \times 10^{-12}$& 18.92\\
& 10000 & 20 & $3.04\times  10^{-8}$ & 24.09 \\
& & 30 & $2.76 \times 10^{-11}$& 30.83\\
& & 40 & $1.68 \times 10^{-11}$& 38.72 \\

\hline
${\cal L}_2(u)$  & 3600 & 20 & $2.29\times 10^{-12}$ & 5.28\\ 
 & & 30  & $8.01 \times 10^{-13}$ & 7.65 \\
  &  & 40&  $4.41 \times  10^{-13}$& 8.58 \\
& 6400 & 20 & $2.40\times  10^{-10}$ & 14.50 \\
& & 30 & $4.01 \times 10^{-12}$& 18.81\\
& & 40 & $3.39 \times 10^{-13}$& 21.74\\
& 10000 & 20 & $1.20\times  10^{-12}$ & 35.59 \\
& & 30 & $1.29 \times 10^{-12}$& 41.35\\
& & 40 & $7.63 \times 10^{-12}$& 49.27 \\
\hline
\end{tabular}
\end{center}
\end{table}

{\bf Example 5.}  For the second experiment, we determine approximations of the matrix functions $log(A_1+I_n) A_1^{-1}B$ and $log(A_2+I_n)A_2^{-1}B$, where $A_1 = \textit{tridiag}(1, 2, 1) \in \mathbb{R}^{n \times n}$ and $A_2 \in \mathbb{R}^{n \times n}$ is the block diagonal matrix  with $2 \times 2$ blocks of the form  

$$\left [ \begin{array}{c c}
a_i &  c \\
c & a_i \\
\end{array}\right],$$

where $c=1/2$ and $a_i=(2i - 1)/(n+1)$ for $i=1,\ldots, n/2$. The approximation errors and CPU times are listed in table \ref{tab4}.

\begin{table}[!h]
\begin{center}
\caption{$H_{\infty}$ exact error of $log(A_1+I_n)A_1^{-1}B$ and $log(A_2+I_n)A_2^{-1}B$  for multiple values of n.}\label{tab4}
\begin{tabular}{l|c|c|c|c}
\hline
\quad \quad \quad \quad Test problem & \;\; n & \;\;  dim.(m)   & \;\;  $ \Vert f(A)B-f_m \Vert _{\infty} $ & \;\; time (s) \\  
\hline
$A_1$  & 2500 & 20 & $1.52 \times 10^{-7}$ & 2.09\\ 
 & & 30  & $2.04 \times 10^{-8}$ & 2.69 \\

& 5000 & 20 & $6.20  \times  10^{-8}$ & 7.24 \\
& & 30 & $8.63 \times 10^{-10}$& 8.20\\

& 7500 & 20 & $3.22\times  10^{-8}$ & 20.05 \\
& & 30 & $6.57 \times 10^{-8}$& 21.35\\

& 10000 & 20 & $1.22 \times  10^{-7}$ & 46.89 \\
& & 30 & $1.56 \times 10^{-9}$& 47.36\\

\hline
$A_2$  & 2500 & 20 & $4.21  \times 10^{-8}$ & 1.37\\ 
 & & 30  & $1.02 \times 10^{-10}$ &1.70 \\

& 5000 & 20 & $7.65  \times  10^{-7}$ & 5.39 \\
& & 30 & $9.62 \times 10^{-9}$&  5.98\\

& 7500 & 20 & $5.10\times  10^{-7}$ & 14.08 \\
& & 30 & $7.32\times 10^{-11}$& 14.80\\

& 10000 & 20 & $3.75\times  10^{-8}$ & 30.70 \\
& & 30 & $1.96 \times 10^{-10}$& 31.82\\

\hline
\end{tabular}
\end{center}
\end{table}

{\bf Example 6.}  For the last experiment, we consider the same matrices  $A=A_1$ and $A=A_2$ using in the previous example to  approximate the matrix function $exp(A)$.

We give the  exact error $\Vert f(A)B-f_m \Vert _{\infty}$ and the required CPU time in table \ref{tab5} for different values of n.

\begin{table}[!h]
\begin{center}
\caption{$H_{\infty}$ exact error of $e^A$.}\label{tab5}
\begin{tabular}{l|c|c|c|c}
\hline
\quad \quad \quad \quad Test problem & \;\; n & \;\;  dim.(m)   & \;\;  $ \Vert f(A)B-f_m \Vert _{\infty} $ & \;\; time (s) \\  
\hline
$A_1$  & 2500 & 20 & $1.17 \times 10^{-9}$ & 1.55\\ 
 & & 30  & $2.10 \times 10^{-9}$ & 2.06 \\

& 5000 & 20 & $3.13  \times  10^{-9}$ & 6.67 \\
& & 30 & $3.44 \times 10^{-10}$&7.33\\

& 7500 & 20 & $4.33 \times  10^{-9}$ & 18.12 \\
& & 30 & $1.36 \times 10^{-8}$ & 19.45\\

& 10000 & 20 & $5.96  \times  10^{-10}$ & 42.50 \\
& & 30 & $6.83  \times 10^{-9}$ & 43.61\\

\hline
$A_2$  & 2500 & 20 & $2.33  \times 10^{-11}$ & 1.08\\ 
 & & 30  & $1.77 \times 10^{-10}$ & 1.40 \\

& 5000 & 20 & $3.85  \times  10^{-11}$ & 4.48 \\
& & 30 & $3.81  \times 10^{-11}$ &  5.25\\

& 7500 & 20 & $1.12 \times  10^{-11}$ & 12.85 \\
& & 30 & $2.50 \times 10^{-10}$  & 13.39\\

& 10000 & 20 & $6.85 \times  10^{-11}$  & 26.39 \\
& & 30 & $4.52 \times 10^{-11}$  &  27.68\\
\hline
\end{tabular}
\end{center}
\end{table}

\section{Conclusion}
In this paper, we introduce a new adaptive rational block Lanczos process, which finds application in approximating the matrix exponential and certain Cauchy-Stieltjes functions. We also derive a new expression for the rational residual error and establish error bound for the exact error. By conducting numerical experiments, we demonstrate the effectiveness of the rational block Lanczos subspace method, showcasing its strong performance. We assess the proposed procedure on benchmark problems of medium and large dimensions, and the results indicate that the adaptive approach yields accurate approximations of small dimension.

\end{document}